\documentclass[b5paper,twoside,headrule,reqno,11pt]{amsart}
\usepackage{latexsym}
\usepackage{amssymb,latexsym,amscd,amsmath,amsthm}

\newtheorem{thm}{Theorem}
\newtheorem{theorem}[thm]{Theorem}

\newtheorem{prop}{Proposition}

\def\({\left(}
\def\){\right)}
\def\[{\left[}
\def\]{\right]}
\def\<{\langle}
\def\>{\rangle}

\begin{document}

\def\evenhead{{\protect\centerline{\textsl{\large{K. Soundararajan}}}\hfill}}

\def\oddhead{{\protect\centerline{\textsl{\large{
An asymptotic expansion related to the Dickman function}}}\hfill}}

\pagestyle{myheadings} \markboth{\evenhead}{\oddhead}

\begin{center}
\Large {\bf An asymptotic expansion related to the Dickman function }
\end{center}
\bigskip

\begin{center}
\large   K. Soundararajan
\end{center}

\let\thefootnote\relax\footnotetext{ 
The author was supported in part by NSF grant DMS-0500711.}

\medskip
\begin{center}
\begin{minipage}{4.5in}
\end{minipage}
\end{center}


\bigskip
  
  In a recent paper Broadhurst \cite{B} considered a generalized class of polylogarithms 
related to the Dickman function.  In exploring these polylogarithms he was 
led to define a sequence of constants $C_k$ which he called the Dickman constants.  
Based on numerical computations, Broadhurst conjectured that the generating 
function $\sum_{k=0}^{\infty} C_k z^k$ equals 
$\exp(\gamma z)/\Gamma(1-z)$, 
so that the constants $C_k$ are related to values of the Riemann  zeta-function at 
integers.   In this note we establish Broadhurst's conjecture.

First we recall that the Dickman function $\rho:\ [0,\infty) \to [0,1]$ is defined by 
$\rho(u) = 1$ for $0\le u\le 1$, and for $u\ge 1$ is given by the unique solution 
to the differential difference equation $u\rho^{\prime}(u) = -\rho(u-1)$.   The Dickman 
function arises naturally in number theory as follows:  the number of 
integers below $x$ all of whose prime factors are below $x^{1/u}$ is 
asymptotic to $\rho(u) x$ as $x\to \infty$.  It is not hard to show that 
$$ 
\rho(u) = \sum_{k=0}^{\infty} \frac{(-1)^{k}}{k!} I_k(u), 
$$ 
where 
$$ 
I_k(u) = \int_{{t_1, \ldots t_k \ge 1}\atop{ t_1+\ldots+t_k \le u}} \frac{dt_1}{t_1} \cdots 
\frac{dt_k}{t_k}. 
$$ 
Note that $I_0(u)=1$ for all $u$, $I_1(u)=0$ for $0\le u\le 1$ and $I_1(u)= \log u$ for 
$1\le u$, and so on.  The sum over $k$ in the formula for $\rho(u)$ is 
in fact a finite sum since $I_k(u)=0$ if $k\ge u$.

Broadhurst uses a different formulation.  He sets $F(\alpha) =\rho(1/\alpha)$ 
which then satisfies $F(\alpha)=1$ for $\alpha\ge 1$, and for $0<\alpha <1$ 
satisfies the differential equation 
$$
F^{\prime}(\alpha) =  \frac{1}{\alpha} F\Big(\frac{\alpha}{1-\alpha}\Big). 
$$ 
He writes $F(\alpha) =\sum_{k=0}^{\infty} L_k(\alpha)$ 
where $L_0(\alpha)=1$ for all $\alpha$, and the ``Dickman polylogarithm" $L_k$ is defined recursively 
by $L_k(t) =0$ for $t\ge 1/k$ and for $t<1/k$ 
$$ 
L_k(t) = -\int_{t}^{1/k} L_{k-1} \Big(\frac{x}{1-x}\Big) \frac{dx}{x}.
$$ 
We may check easily that 
$$ 
L_k(\alpha) = \frac{(-1)^k}{k!} I_k(1/\alpha). 
$$ 
Broadhurst conjectures that there exists a sequence of constants $C_k$ 
(for integers $k\ge 0$) such that, as $t \to 0$ we have 
$$ 
L_k(t) = \sum_{j=0}^{k} {C_j} \frac{(\log t)^{k-j}}{(k-j)!} +o(1). 
$$ 
Equivalently, in terms of $I_k(u)$, Broadhurst's conjecture reads, as $u\to \infty$, 
$$ 
I_k(u) = \sum_{j=0}^k \binom{k}{j} D_j (\log u)^{k-j} +o(1),
$$ 
with $D_j = (-1)^j j! C_j$.  
He found that $C_0=1$, $C_1=0$, $C_2=-\pi^2/12$, $C_3=-\zeta(3)/3$, 
and formulated the elegant conjecture that 
$$ 
\sum_{k=0}^{\infty} C_k z^k = \frac{e^{\gamma z}}{\Gamma(1-z)}. 
$$ 
We establish this conjecture below.

\begin{prop}  For natural numbers $k\ge 0$ define the 
constants 
$$ 
C_k = \frac{1}{k!} \frac{1}{2\pi i } \int_{c-i\infty}^{c+i\infty} \frac{e^s}{s} (\log s +\gamma)^k ds, 
$$ 
where $c>0$ and the integral converges conditionally and is interpreted as 
$\lim_{T\to \infty} \int_{c-iT}^{c+iT}$.   We have, for all complex $z$,  
$$ 
\sum_{k=0}^{\infty} C_k z^k = \frac{e^{\gamma z}}{\Gamma(1-z)}.  
$$ 
\end{prop}
\begin{proof}  First we rewrite the definition of $C_k$ by integrating by parts.  Thus 
\begin{equation}
\label{alt}
C_k = \frac{1}{k! } \frac{1}{2\pi i} \int_{c-i\infty}^{c+i\infty} e^{s} \Big( 
\frac{(\log s+\gamma)^k}{s^2} -\frac{k(\log s+\gamma)^{k-1}}{s^2} \Big) ds.  
\end{equation}
This integral is now absolutely convergent.  Moreover, choosing $c=1$, 
we obtain the bound 
$$ 
|C_k| \le \frac{1}{2\pi k!} \int_{-\infty}^{\infty} \frac{e (1+k) (\log (1+|t|) + \gamma +\pi/2)^k}{1+t^2 } dt 
\le C (k+1) 
$$
for some constant $C$.  Thus the series $\sum_{k=0}^{\infty} C_k z^k$ 
converges absolutely for $|z|<1$ and defines in that region an analytic function.   

Using our expression \eqref{alt} we obtain that for $|z|<1$  
$$
\sum_{k=0}^{\infty} C_k z^k 
= \frac{1}{2\pi i} \int_{c-i\infty}^{c+i\infty} e^s  \frac{e^{\gamma z}s^z}{s^2}(1-z) ds  = \frac{e^{\gamma z}(1-z)}{\Gamma(2-z)} =\frac{e^{\gamma z}}{\Gamma(1-z)}.
$$  
The relation connecting the contour integral to $1/\Gamma(2-z)$ may be derived 
by deforming the line of integration to a contour traveling just above the real 
axis from $-\infty$ to zero, taking a detour around zero, and then traveling back to $-\infty$ 
just below the real axis.  The resulting integral is Hankel's contour integral for 
the reciprocal of the $\Gamma$-function; see for example Chapter IX of 
\cite{C}, and in particular the ``miscellaneous example" 24 which attributes our formula above to Laplace.   

Since $e^{\gamma z}/\Gamma(1-z)$ is analytic for all $z\in {\Bbb C}$, we 
conclude that the series $\sum_{k=0}^{\infty} C_k z^k$ converges absolutely 
for all $z$ and equals $e^{\gamma z}/\Gamma(1-z)$.  
\end{proof}

\begin{theorem}  Setting $D_j = (-1)^j j! C_j$ we have as $u\to \infty$ 
$$ 
I_k(u) = \sum_{j=0}^{k} \binom{k}{j} D_j (\log u)^{k-j}  +O\Big(\frac{(\log u)^k}{u}\Big)
 $$ 
and, equivalently 
$$ 
L_k(t) = \sum_{j=0}^{k} C_j \frac{(\log t)^{k-j}}{(k-j)!} + O(t(|\log t|)^k). 
$$ 
\end{theorem} 

\begin{proof}
Recall Perron's formula which gives, for any $c>0$,  that 
$$ 
\frac{1}{2\pi i} \int_{c-i\infty}^{c+i\infty} \frac{e^{\lambda s}}{s} ds = \begin{cases}
1&\text{if  } \lambda >0 \\ 
0&\text{if  }\lambda < 0,\\
\end{cases}
$$ 
and for completeness we note that the integral is $1/2$ if $\lambda =0$.  Note that
 the integral above is only conditionally convergent and should be understood as 
 $\lim_{T\to \infty} \int_{c-iT}^{c+iT}$.   Therefore we find that 
 $$ 
 I_k(u) = \frac{1}{2\pi i } \int_{c-i\infty}^{c+i\infty} \frac{e^{s}}{s} \int_{{t_1, \ldots, t_k \ge 1}} 
 \frac{e^{-(t_1+\ldots+t_k)s/u}}{t_1\cdots t_k} dt_1 \cdots dt_k \ ds.  
 $$
 Making a change of variables in the inner integrals we obtain 
 that 
 \begin{equation}
 \label{I1} 
 I_k(u) = \frac{1}{2\pi i} \int_{c-i\infty}^{c+i\infty} \frac{e^s}{s} \Big(\int_{1/u}^{\infty} 
 \frac{e^{-ts}}{t} dt \Big)^k ds.
 \end{equation} 
 
 Now observe that 
 $$ 
 \int_{1/u}^{\infty} \frac{e^{-ts}}{t}dt 
 = -\int_0^{1} \frac{1-e^{-ts}}{t} dt +\int_{1}^{\infty} \frac{e^{-ts}}{t} dt 
 +\int_0^{1/u} \frac{1-e^{-ts}}{t} dt + \log u. 
 $$ 
 For $s$ with positive real part we may see that 
 $$ 
 \int_0^1 \frac{1-e^{-ts}}{t} dt -\int_1^{\infty} \frac{e^{-ts}}{t}dt =\log s +\gamma, 
 $$ 
 and thus we conclude that 
 \begin{equation} 
 \label{formula} 
 \int_{1/u}^{\infty} \frac{e^{-ts}}{t} dt 
 = \log u -\log s -\gamma + \int_0^{1/u} \frac{1-e^{-ts}}{t} dt. 
 \end{equation}
 
Write $G(u,s) = \int_0^{1/u} (1-e^{-ts})/t dt$.   Using \eqref{formula} in \eqref{I1} we find that 
 \begin{equation}
 \label{I2}
 I_k(u) = \sum_{j=0}^{k} \binom{k}{j} 
 \frac{1}{2\pi i} \int_{c-i\infty}^{c+i\infty} \frac{e^s}{s} G(u,s)^j (\log u -\log s -\gamma)^{k-j} ds. 
 \end{equation}
Using the binomial theorem and our definition of the 
constants $C_j$ in Proposition 1 we see that the term $j=0$ equals 
\begin{equation}
\label{j} 
\frac{1}{2\pi i} \int_{c-i\infty}^{c+i\infty} \frac{e^s}{s} (\log u -\log s -\gamma)^k ds =
\sum_{\ell =0}^{k} \binom{k}{\ell} D_\ell (\log u)^{k-\ell}.
\end{equation} 

It remains finally to estimate the terms with $j\neq 0$.  Integration by parts shows that 
the $j$-th term is 
$$ 
- \frac{1}{2\pi i} \int_{c-i\infty}^{c+i\infty} 
e^s \frac{d}{ds} \Big( \frac{G(u,s)^j}{s}(\log u -\log s-\gamma)^{k-j}\Big) ds.   
$$ 
Since, for Re$(s)>0$,  
$$ 
|G(u,s)| =O\Big( \int_0^{1/u} \min \Big(|s|,\frac{1}{t}\Big) dt \Big)= O( \log (1+|s|/u)),
$$  
and 
$$
\Big| \frac{d}{ds} G(u,s)\Big| =\frac{ |1-e^{-s/u}|}{|s|} = O\Big( \min\Big(\frac{1}{|s|}, \frac 1u\Big)\Big), 
$$ 
we may bound our integrand above by (choosing $c=1$ and writing $s=c+it$)
\begin{eqnarray*}
\frac{(1+\log u +\log (1+|t|))^{k-j}}{1+|t|} 
&&\Big(\log (1+(1+|t|)/u))^{j-1} \min\Big(
\frac{1}{1+|t|},\frac 1u\Big) \\
&& + \frac{(\log (1+(1+|t|)/u))^j}{1+|t|}    \Big).
\\
\end{eqnarray*}
Integrating this over $t$ from $-\infty$ to $\infty$ we conclude that the $j$-th 
term is $O((\log u)^k/u)$ as desired.  This proves our Theorem. 
\end{proof}

We take this opportunity to make an historical observation on the 
Dickman function.  While Dickman's paper \cite{D} appears to be the first 
published account of the $\rho$-function, during the Focused Week on Quadratic Forms 
and Theta Functions held in March 2010 at the University of Florida, I chanced upon the 
following entry in Ramanujan's 
unpublished papers:   see page 337 of \cite{R}.  
We quote the second half of this page: 

$\phi(x)$ {\sl is the no. of nos of the form} 
$$ 
2^{a_2} \cdot 3^{a_3} \cdot 5^{a_5}\cdots p^{a_p} \qquad p \le x^{\epsilon}
$$
{\sl not exceeding $x$}.  
$$ 
\tfrac 12\le \epsilon \le 1, \qquad \phi(x) \sim x \Big\{ 1 -\int_\epsilon^1 \frac{d\lambda_0}{\lambda_0}\Big\}
$$
$$ 
\tfrac 13\le \epsilon \le \tfrac 12, \qquad \phi(x) \sim x \Big\{ 1-\int_{\epsilon}^1 \frac{d\lambda_0}{\lambda_0} + \int_{\epsilon}^{\frac 12} \frac{d\lambda_1}{\lambda_1} \int_{\lambda_1}^{1-\lambda_1}\frac{d\lambda_0}{\lambda_0}\Big\}
$$
\begin{eqnarray*}
\tfrac 14 \le \epsilon \le \tfrac {1}{3}, \qquad \phi(x)\sim x \Big \{ 1&&- \int_{\epsilon}^1 \frac{d\lambda_0}{\lambda_0} + \int_{\epsilon}^{\frac 12} \frac{d\lambda_1}{\lambda_1} \int_{\lambda_1}^{1-\lambda_1}\frac{d\lambda_0}{\lambda_0} \\
&&- \int_{\epsilon}^{\frac 13} \frac{d\lambda_2}{\lambda_2} \int_{\lambda_2}^{\frac{1-\lambda_2}{2}} 
\frac{d\lambda_1}{\lambda_1} \int_{\lambda_1}^{1-\lambda_1} \frac{d\lambda_0}{\lambda_0}\Big\} \\
\end{eqnarray*}
\begin{eqnarray*}
\tfrac 15 \le \epsilon \le \tfrac 14, \qquad \phi(x) \sim x &&\Big\{ 1- \int_{\epsilon}^1 \frac{d\lambda_0}{\lambda_0} + \int_{\epsilon}^{\frac 12} \frac{d\lambda_1}{\lambda_1} \int_{\lambda_1}^{1-\lambda_1}\frac{d\lambda_0}{\lambda_0} \\
&&- \int_{\epsilon}^{\frac 13} \frac{d\lambda_2}{\lambda_2} \int_{\lambda_2}^{\frac{1-\lambda_2}{2}} 
\frac{d\lambda_1}{\lambda_1} \int_{\lambda_1}^{1-\lambda_1} \frac{d\lambda_0}{\lambda_0}\\
&&+\int_{\epsilon}^{\frac 14}
\frac{d\lambda_3}{\lambda_3} \int_{\lambda_3}^{\frac{1-\lambda_3}{3}} \frac{d\lambda_2}{\lambda_2} 
\int_{\lambda_2}^{\frac{1-\lambda_2}{2}} \frac{d\lambda_1}{\lambda_1} \int_{\lambda_1}^{1-\lambda_1} \frac{d\lambda_0}{\lambda_0}\Big\}\\
\end{eqnarray*}
{\sl and so on.} 

Digressing from the topic of this note, we point out that the first half of this 
page of Ramanujan is also of interest to number theorists.   Here Ramanujan 
observes that for fixed $k$ the number of integers below $x$ with at most 
$k$ prime factors is asymptotically 
$$ 
\frac{x}{\log x} \Big( 1+ \log \log x + \frac{(\log \log x)^{2}}{2!} + \ldots + \frac{(\log \log x)^{k-1}}{(k-1)!}\Big). 
$$ 
He notes that this formula also holds ``when $k$ is infinite."  Ramanujan then 
asks whether the formula holds ``when $k$ is a function of $x$".   This interesting 
question on the uniformity with which such an asymptotic formula 
holds was settled by the work of Sathe \cite{Sa} and Selberg \cite{Se}.  

{\bf Acknowledgments}.  I am grateful to Jeff Lagarias for drawing my attention to this problem.

\bigskip

\noindent
Department of Mathematics, Stanford University, Stanford, CA 94305, USA\\
E-mail: ksound@math.stanford.edu\\

\end{document}